\documentclass[9pt,twoside]{IEEEtran}
\usepackage{cite}

\usepackage{cprotect}

\ifCLASSINFOpdf
  \usepackage[pdftex]{graphicx}
  \graphicspath{{../pdf/}{../jpeg/}}
  \DeclareGraphicsExtensions{.pdf,.jpeg,.png}
\else
  \usepackage[dvips]{graphicx}
 \graphicspath{{../eps/}}
  \DeclareGraphicsExtensions{.eps}
\fi

\usepackage[cmex10]{amsmath}
\usepackage{amssymb}
\usepackage{mathtools}
\usepackage{amsthm}
\usepackage[ruled]{algorithm2e}

\usepackage{array}

\ifCLASSOPTIONcompsoc
\usepackage[caption=false,font=normalsize,labelfont=sf,textfont=sf]{subfig}
\else
\usepackage[caption=false,font=footnotesize]{subfig}
\fi
\usepackage{fixltx2e}
\usepackage{url}

\usepackage{mathtools}

\usepackage{setspace}
\usepackage{lineno}
\usepackage{xcolor}
\def\HiLi{\leavevmode\rlap{\hbox to \hsize{\color{yellow!50}\leaders\hrule height .8\baselineskip depth .5ex\hfill}}}

\usepackage{changepage} 

\DeclareMathOperator*{\argmax}{\arg\!\max}
\newtheorem{lemma}{Lemma}

\begin{document}
%
\title{eOMP: Finding Sparser Representation by Recursively Orthonormalizing the Remaining Atoms}
%
%
%
\author{Yuanyi~Xue, and~Yao~Wang
\thanks{Y. Xue, and Y. Wang are with the Department
of Electrical and Computer Engineering, Polytechnic School of Engineering, New York University. E-mail: \{yxue,yw523\}@nyu.edu.}}
\vspace{-.2in}
\maketitle
\begin{abstract}
Greedy algorithms for minimizing L0-norm of sparse decomposition have profound application impact on many signal processing problems. In the sparse coding setup, given the observations $\mathrm{y}$ and the redundant dictionary $\mathbf{\Phi}$, one would seek the most sparse coefficient (signal) $\mathrm{x}$ with a constraint on approximation fidelity. In this work, we propose a greedy algorithm based on the classic orthogonal matching pursuit (OMP) with improved sparsity on $\mathrm{x}$ and better recovery rate, which we name as eOMP. The key ingredient of the eOMP is recursively performing one-step orthonormalization on the remaining atoms, and evaluating correlations between residual and orthonormalized atoms. We show a proof that the proposed eOMP guarantees to maximize the residual reduction at each iteration. Through extensive simulations, we show the proposed algorithm has better exact recovery rate on i.i.d. Gaussian ensembles with Gaussian signals, and more importantly yields smaller L0-norm under the same approximation fidelity compared to the original OMP, for both synthetic and practical scenarios. The complexity analysis and real running time result also show a manageable complexity increase over the original OMP. We claim that the proposed algorithm has better practical perspective for finding more sparse representations than existing greedy algorithms.
\end{abstract}


%

\section{Introduction}\label{sec:intro}
%
%
%
%
\IEEEPARstart{T}{he} recent developments in sparse recovery have yielded profound impact on many old and new signal processing problems. In the sparse coding setting, people are looking for solving the classic inverse problem where the sampling matrix (or dictionary) $\mathbf{\Phi}$ constructs an underdetermined system. Specifically, given $\mathbf{\Phi} \in \mathbb{R}^{N\times M}$ where $M \gg N$ and observations $\mathrm{y} \in \mathbb{R}^{N}$, it solves an $\mathrm{x} \in \mathbb{R}^{M}$ via L0-norm minimization problem:
\begin{align}\label{eq:l0}
\min \quad &|\mathrm{x}|_0 \nonumber \\
\text{s.t. } \quad & \| \mathrm{y} - \mathbf{\Phi}\mathrm{x}\|_2 < \epsilon 
\end{align}

The nature of non-convexity of the L0-norm leads to two genre of algorithms for solving the above problem. The first category of algorithms seeks to relax the non-convex L0-norm term into a more optimization friendly L1-norm, such as the well-known LASSO algorithm~\cite{tibshirani1996regression}. Although it has mathematical guarantee on global convergence, the L1-norm in fact does not precisely measure the sparsity, which is the number of non-zero coefficients in the solution vector. On the other hand, another group of scholars design a series of primarily greedy algorithms to minimize the L0-norm directly. An exemplar of such algorithms is the orthogonal matching pursuit~\cite{tropp2007signal}, which iteratively finds the best correlated atom to the residual. More recently, several improved algorithms have been proposed. In~\cite{needell2009cosamp} and~\cite{dai2009subspace}, the authors proposed two similar algorithms (CoSaMP and subspace pursuit), respectively, that have been proven to have better exact recovery conditions and faster convergence, for signals with known sparsity.

Although the new algorithms are promising, we would like to note that they are mostly based on the following two assumptions: 1) all above cited work have proved exact recovery conditions that in general require the sampling matrix $\mathbf{\Phi}$ to be nearly orthonormal (as measured by the restricted isometry property~\cite{candes2005decoding}); and 2) in the synthetic simulations demonstrated, the signals are often constructed with known sparsity. This is particularly true for~\cite{needell2009cosamp} and~\cite{dai2009subspace}, wherein the estimated true sparsity is needed as the algorithm input. We would like to note these assumptions hamper the usability of these algorithms in the practical sparse coding scenario, where the sampling matrix is usually highly correlated and the actual signal sparsity is unknown. Indeed, \cite{needell2009cosamp} and~\cite{dai2009subspace} are targeting a different version of Eq.~\ref{eq:l0}, which tries to minimize the fidelity with an equality constraint on the L0-norm of $\mathrm{x}$.

From this standpoint, we propose a greedy algorithm based on the classic OMP algorithm which better suits the scenario with unknown sparsity and correlated sampling matrix. The proposed algorithm recursively performs one-step orthonormalization on the remaining atoms with respect to previously chosen atoms in the dictionary at each iteration, so that the orthnormalized remaining atoms span over the null space of the current signal estimate. We show through extensive synthetic and real simulations that the proposed algorithm is able to achieve better sparsity compared to the classic OMP, with a manageable complexity increase. We name the proposed algorithm as eOMP because it uses embedded recursive orthonormalization.

The rest of the paper is organized as follows. In Sec.~\ref{sec:revOMP}, we lay out the details of the proposed eOMP algorithm and we show the proof that the proposed eOMP guarantees to maximize residual reduction at each iteration; in Sec.~\ref{sec:exp}, we provide extensive simulation results for synthetic simulations under two sampling matrix configurations, the i.i.d Gaussian random ensemble and the overcomplete DCT dictionary, as well as a practical simulation based on coding video blocks by finding a sparse representation over a highly redundant and correlated self-adaptive dictionary~\cite{xue2014video}; we show the complexity analysis and real running time statistics of the proposed algorithm in Sec.~\ref{sec:comp}; we conclude the paper in Sec.~\ref{sec:conc}.
\section{Improving the OMP Algorithm}\label{sec:revOMP}
In the classic OMP, at each iteration, the algorithm consists of the following three major steps: find the most correlated column (atom) among all remaining columns of the sampling matrix with the current residual and append the chosen atom index into the set of chosen atoms (known as the support), a least squares update of the coefficients corresponding to the new support, and at last the residual is updated using least squares coefficients. As all greedy algorithms, the OMP converges when the residual norm stops decreasing or reaches the preset $\epsilon$ shown in Eq.~\ref{eq:l0}.

Before we jump to our proposed algorithm, we introduce an alternative way of implementing the original OMP algorithm by orthonormalizing the chosen atoms incrementally. Specifically, at each iteration, after selecting the chosen atom, we orthonormalize it with respect to the previous chosen atoms using the Gram-Schmidt algorithm. The coefficient with respect to this orthonormalized atom can be then obtained simply by the inner product (i.e. projection)  of the current residual to this orthonormalized atom. This alternative way of implementing OMP is shown in Algorithm~\ref{alg:embedOMP}. We would like to note that although the solution vector corresponds to the orthonormalized atoms, finding the coefficients to the original chosen atoms can be done after all atoms are chosen by a single least squares step.
\begin{algorithm}
 \caption{Alternative OMP algorithm with incremental orthonormalization of chosen atoms}
 \KwData{$\mathbf{\Phi} \in \mathbb{R}^{N\times M}, \mathrm{y} \in \mathbb{R}^{N\times 1}, \epsilon,$ assuming $\mathbf{\Phi}$ has zero mean and unit variance, and $\mathrm{y}$ has zero mean.}
 \KwResult{The set of chosen atom index $\mathcal{K}$, chosen orthonormalized atoms $\mathbf{D}$, solution vector $\mathrm{z}$ w.r.t $\mathbf{D}$, solution vector $\mathrm{x}$ w.r.t $\mathbf{\Phi}$.}
 initialize residual $\mathrm{r} \gets \mathrm{y}$, $\mathcal{K}=\emptyset$, $\mathcal{K}^c=\{1,2,\ldots,M\}$, $\mathbf{\Psi}=\mathbf{\Phi}$, $\mathbf{D}=\mathbf{0}$, iteration counter $t = 1$ \;
 \While{$\| \mathrm{r}\|_2 >= \epsilon$}{
  1. Update remaining atoms $\mathbf{\Psi} \gets \mathbf{\Phi}(:,i)\text{,\ } \forall i \in \mathcal{K}^c$\;
  2. $\mathrm{\phi}_t \gets \argmax_{\mathrm{\phi}_n \in \mathbf{\Psi}} \vert\langle \mathrm{\phi}_n,\mathrm{r} \rangle\vert$\ and $k \gets \text{index of } \mathrm{\phi}_t$\;
  3. Update chosen index set $\mathcal{K} \gets \mathcal{K}\cup \{k\}$ and remaining index set $\mathcal{K}^c \gets \mathcal{K}^c \backslash \{k\}$\;
  4. Orthonormalize $\mathrm{\phi}_t$ w.r.t $\mathbf{D}$: $\mathrm{\psi}_t \gets \mathrm{\phi}_t \perp \mathbf{D}$\;
  5. Augment $\mathbf{D} \gets [\mathbf{D}|\mathrm{\psi}_t]$\;
  6. $\mathrm{z}_t \gets \langle\mathrm{r},\mathrm{\psi}_t \rangle$\;
  7. Update $\mathrm{r} \gets \mathrm{r}-\mathrm{z}_t\mathrm{\psi}_t$\;
  8. Augment $\mathrm{z} \gets [\mathrm{z}|\mathrm{z}_t]$\;
  9. $t \gets t + 1$\;
 }
 (\textit{Optional}) Denote the chosen original atoms indexed by $\mathcal{K}$ as $\mathbf{\Phi}_{\mathcal{K}}$, solve least square problem $\mathrm{x} = \mathbf{\Phi}_{\mathcal{K}}^{\dagger} \mathrm{y}$.
 \vspace{.1in}
\label{alg:embedOMP}
\end{algorithm} 

We note that the solution vector $\mathrm{z}$ can be written as 
\begin{equation}\label{eq:ztemomp}
\mathrm{z}=\mathbf{D}^T \mathrm{y}.
\end{equation}
This is because $\mathrm{z}_t = \langle \mathrm{r}_t, \mathrm{\psi}_t\rangle$ in Step 6 can also be written as $\mathrm{z}_t= \langle y,\mathrm{\psi}_t \rangle$ because $\mathrm{y}$ can be represented as linear combination of previously found orthonormalized atoms (which are all orthogonal to $\mathrm{\psi}_t$) and $\mathrm{r}_t$. 

To show Algorithm~\ref{alg:embedOMP} yields identical solution to the original OMP, let us examine the QR decomposition method commonly used for the least squares update. For a residual $\mathrm{r} = \mathrm{y} - \mathbf{\Phi}_{\mathcal{K}}\mathrm{x}$, denote the QR decomposition on the chosen atom matrix $\mathbf{\Phi}_{\mathcal{K}}$ by $\mathbf{Q}\mathbf{R}$. The $\mathbf{Q}$ matrix is exactly the $\mathbf{D}$ matrix in Algorithm~\ref{alg:embedOMP} by the construction of the QR decomposition. To show that the original OMP solution is equivalent to Algorithm~\ref{alg:embedOMP}, we just need to show that $\mathbf{R}\mathrm{x}$ is exactly equal to $\mathrm{Z}$ from Algorithm 1, when $\mathbf{Q}=\mathbf{D}$. Left multiply $\mathbf{Q}^T$ on the residual equation, we get the following
\begin{align}\label{eq:qr}
\mathbf{Q}^T\mathrm{r} & = \mathbf{Q}^T\mathrm{y} - (\mathbf{Q}^T\mathbf{Q})\mathbf{R}\mathrm{x} \nonumber \\
& = \mathbf{Q}^T\mathrm{y} - \mathbf{R}\mathrm{x}
\end{align}

To minimize $\|r\|_2$ (least square), it is equivalent to solve the system of linear equations by setting Eq.~\ref{eq:qr} to zero.
\begin{equation}
\min \|r\|_2 \Leftrightarrow \mathbf{R}\mathrm{x} = \mathbf{Q}^T\mathrm{y}
\end{equation}
The right side of the above equation is identical to the solution in Eq.~\ref{eq:ztemomp}. Therefore, Algorithm~\ref{alg:embedOMP} is an alternative way to give the same solution as the original OMP.

Our proposed eOMP algorithm is motivated by Algorithm~\ref{alg:embedOMP}.  Notice that the update in Step 7 of Algorithm~\ref{alg:embedOMP} will lead to a residual reduction norm of $|\mathrm{z}_t|$ with each new atom. Therefore, to maximize the residual reduction by each new atom, we should maximize $|\mathrm{z}_t|$, the magnitude of the correlation of the current residual with the othonormalized chosen atom. However, in Step 2, the atom is chosen based on the correlation with the original atoms. Only after an atom is chosen based on this correlation, then the atom is orthonormalized (Step 4), and the coefficient is obtained using the projection to the orthonormalized atom (Step 6). 

To maximize the error reduction at each new iteration, eOMP algorithm first orthonormalizes all the remaining atoms with respect to all previously chosen atoms, and then compute the correlation of the current residual with these othonormalized atoms and find the one with the largest correlation magnitude. Instead of orthonormalizing each remaining original atom with respect to all previously chosen atoms, we further propose to perform the orthonormalization recursively, i.e., each time a new atom is chosen, all remaining atoms are orthonormalized with respect to this new atom only, and the original remaining atoms will be replaced with the newly othonormalized ones.

More specifically, starting at the second iteration, the algorithm recursively performs a one-step orthonormalization of all remaining atoms with respect to the chosen atoms. Denote the last orthonormalized atom in the chosen set $\mathbf{D}$ by $\mathrm{d}$, the one-step orthonormalization operates on each remaining atom $\mathrm{\psi}_i$ by:
\begin{align}\label{eq:onestepGS}
& \mathrm{\psi}_i \gets \mathrm{\psi}_i - \langle \mathrm{d},\mathrm{\psi}_i \rangle \mathrm{d} \nonumber \\
& \mathrm{\psi}_i \gets \frac{\mathrm{\psi}_i}{\lVert\mathrm{\psi}_i\rVert_2}
\end{align}
The orthonormalized atoms, instead of their original atoms in the sampling matrix, are then used to evaluate the correlation with the current residual and the one with the highest correlation (in magnitude) is chosen as the next atom. The computation of the new coefficient and new residual can be done using the newly chosen atom, which is already orthonormalized.

The proposed algorithm is presented in Algorithm~\ref{alg:OMP}. We highlight the differences of our proposed eOMP from Algorithm~\ref{alg:embedOMP}.
\begin{algorithm}
 \caption{Proposed eOMP algorithm with recursive orthonormalization of remaining atoms}
 \KwData{$\mathbf{\Phi} \in \mathbb{R}^{N\times M}, \mathrm{y} \in \mathbb{R}^{N\times 1}, \epsilon,$ assuming $\mathbf{\Phi}$ has zero mean and unit variance, and $\mathrm{y}$ has zero mean.}
 \KwResult{The set of chosen atom index $\mathcal{K}$, chosen orthonormalized atoms $\mathbf{D}$, solution vector $\mathrm{z}$ w.r.t $\mathbf{D}$, solution vector $\mathrm{x}$ w.r.t $\mathbf{\Phi}$.}
 initialize residual $\mathrm{r} \gets \mathrm{y}$, $\mathcal{K}=\emptyset$, $\mathcal{K}^c=\{1,2,\ldots,M\}$, $\mathbf{\Psi}=\mathbf{\Phi}$, $\mathbf{D}=\mathbf{0}$, iteration counter $t = 1$ \;
 \While{$\| \mathrm{r}\|_2 >= \epsilon$}{
  1. Update remaining atoms $\mathbf{\Psi} \gets \mathbf{\Psi}(:,i)\text{,\ } \forall i \in \mathcal{K}^c$\;
  \HiLi 2. Perform one-step orthonormalization Eq.~\ref{eq:onestepGS} of all columns \HiLi in $\mathbf{\Psi}$ w.r.t $\mathbf{\psi}_{t-1}$\;
  \HiLi 3. $\mathrm{\psi}_t \gets \argmax_{\mathrm{\psi}_n \in \mathbf{\Psi}} \vert\langle \mathrm{\psi}_n,\mathrm{r} \rangle\vert$\ and $k \gets \text{index of } \mathrm{\psi}_t$\;
  4. Update chosen index set $\mathcal{K} \gets \mathcal{K}\cup \{k\}$ and remaining index set $\mathcal{K}^c \gets \mathcal{K}^c \backslash \{k\}$\;
  5. Augment $\mathbf{D} \gets [\mathbf{D}|\mathrm{\psi}_t]$\;
  6. $\mathrm{z}_t \gets \langle\mathrm{r},\mathrm{\psi}_t \rangle$\;
  7. Update $\mathrm{r} \gets \mathrm{r}-\mathrm{z}_t\mathrm{\psi}_t$\;
  8. Augment $\mathrm{z} \gets [\mathrm{z}|\mathrm{z}_t]$\;
  9. $t \gets t + 1$\;
 }
 (\textit{Optional}) Denote the chosen original atoms indexed by $\mathcal{K}$ as $\mathbf{\Phi}_{\mathcal{K}}$, solve least square problem $\mathrm{x} = \mathbf{\Phi}_{\mathcal{K}}^{\dagger} \mathrm{y}$.
 \vspace{.1in}
\label{alg:OMP}
\end{algorithm}

There are a few implementation notes to accelerate the proposed algorithm. First, the final least squares step is optional, needed only when a solution vector with respect to the original sampling matrix $\mathbf{\Phi}$ is desired. Computationally this is equivalent to the last iteration solution update in the original OMP. Note that in applications of sparse representation for signal compression, it is better to use the coefficients associated with orthonormalized atoms to represent a signal, as in~\cite{xue2014video}. Second, in Step 3 when selecting the best atom for the current iteration, one could also mark the atoms that have near-zero correlations and exclude them from the consideration in future iterations. Last, one could impose additional convergence conditions in the while-loop. One commonly used condition is to compare the norm of $\mathrm{r}$ in the most recent iterations and terminate when the norm stops to decrease significantly.

Now we turn to giving a proof on why eOMP works better. For that, we have the following lemma.
\begin{lemma}
In each iteration, eOMP maximizes the residual reduction $\|\Delta\|_2$, where $\Delta\coloneqq\mathrm{r}_{t-1} - \mathrm{r}_t$.
\end{lemma}
 
\begin{proof}
Given an orthonormalized atom $\mathrm{\psi}$, the residual reduction $\Delta$ from using that atom is given by $\alpha\mathrm{\psi}$, where $\alpha = \langle \mathrm{r}_{t-1}, \mathrm{\psi}\rangle$. The residual norm is $\|\Delta\|_2 = |\alpha|$.

In eOMP, Step 3 selects a particular orthonormalized atom to exactly maximize $|\alpha|$, since $\mathrm{\psi}_t = \argmax_{\mathrm{\psi}} \langle \mathrm{r}_{t-1}, \mathrm{\psi}\rangle = \alpha$. Therefore eOMP guarantees to maximize the residual reduction.
\end{proof}
\section{Numerical Experiments}\label{sec:exp}
In this section, we show our numerical experiments in both synthetic and real settings for the proposed eOMP algorithm and compare its performance with original OMP. We first show and compare the exact recovery rate of the signal $\mathrm{x}$ with known sparsity, and then the recovered sparsity when the estimate of the observation $\hat{\mathrm{y}}$ is numerically identical to $\mathrm{y}$, under unknown sparsity conditions.
\subsection{Synthetic simulations}\label{sec:exp_syn}
In the first configuration, we consider the exact recovery rate experiment. Consider an i.i.d. Gaussian ensemble $\mathbf{\Phi} \in \mathbb{R}^{N\times 2N}$, with $N = 128$, we construct a $k$-sparse Gaussian signal $\mathrm{x} \in \mathbb{R}^{N}$ and produce the observations $\mathrm{y} = \mathbf{\Phi}\mathrm{x}$. We use our proposed eOMP algorithm and the original OMP algorithm to run $k$ iterations, and compare the exact recovery rate by measuring whether the recovered signal $\hat{\mathrm{x}}$ and $\mathrm{x}$ are numerically identical. The above simulation is ran for 500 realizations.

Figure~\ref{fig:exactrecov} compares the exact recovery rate between eOMP and OMP.  Proposed eOMP algorithm still outperforms the original OMP and has a slower dropping rate beyond the critical sparsity point (i.e. the point where algorithm starts to fail exactly recovering the signals).
\begin{figure}
\centering
\includegraphics[width=0.85\linewidth]{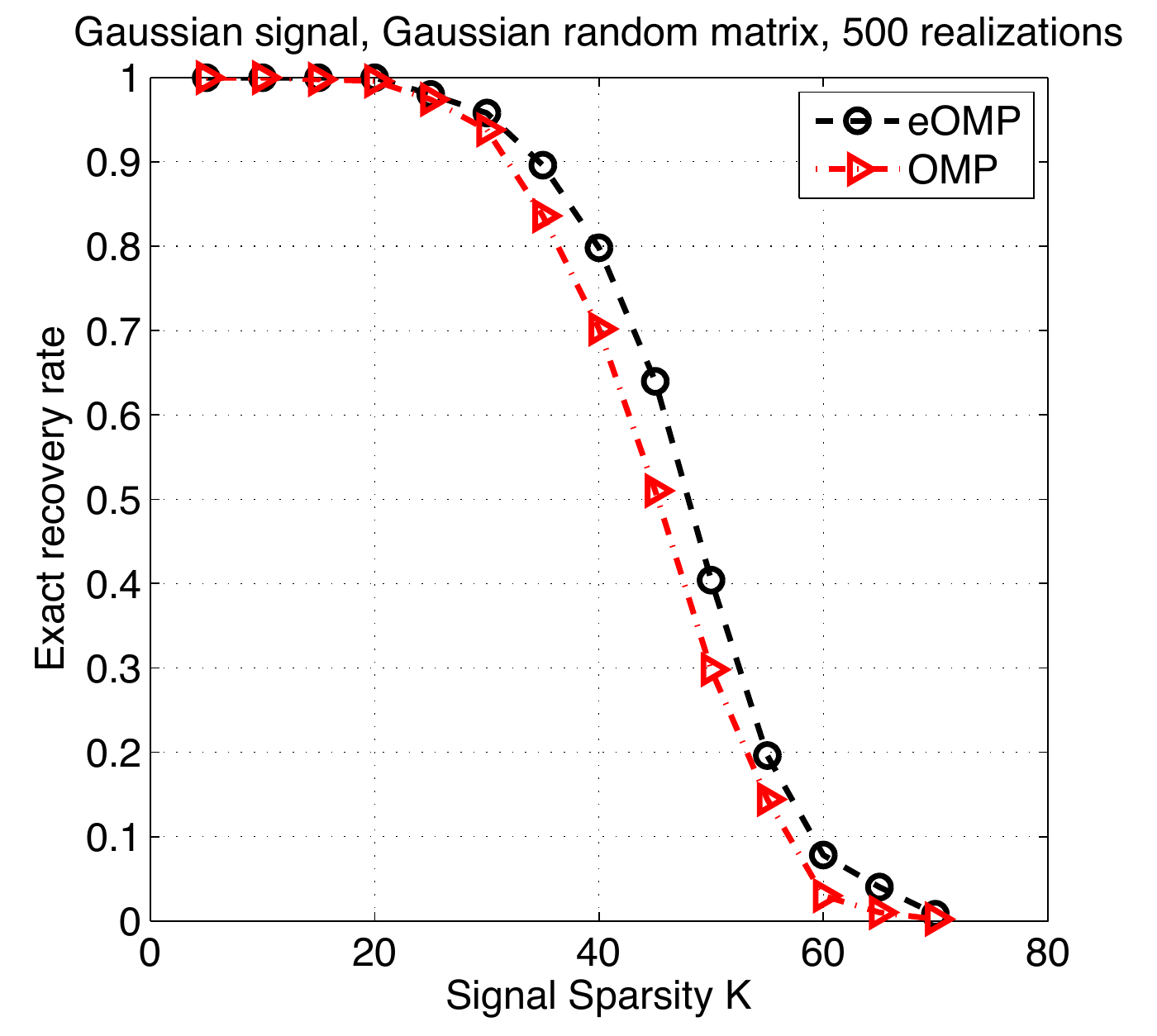}
\caption{Exact recovery rate for Gaussian sampling matrix}
\label{fig:exactrecov}
\end{figure}

In the second configuration of synthetic signals, we consider the unknown sparsity case. To evaluate, given a sampling matrix $\mathbf{\Phi}$ and $k$, we run both the eOMP and the original OMP until the residual norm $\|\mathrm{r}\|_2$ reaches machine precision ($\epsilon < 10^{-5}$). We then measure the L0-norm of recovered signals $\hat{\mathrm{x}}$, which we term as the recovered sparsity. We consider two types of sampling matrix, the i.i.d. Gaussian ensemble of $\mathbb{R}^{N\times 2N}$ and the overcomplete DCT (ODCT) with $2\times$, $4\times$, and $8\times$ number of signal dimension. Note that the ODCT atoms have higher mutual coherence than the i.i.d. Gaussian ensemble. The above simulation is also ran for 500 realizations for each $k$.

Figure~\ref{fig:sparserecov_gauss} shows the recovered sparsity by eOMP and OMP for i.i.d. Gaussian ensembles. Our simulations show that eOMP results in lower recovered K than OMP over the entire range of true signal sparsity K examined (1 to 80). However, to show the difference more clearly, we only show the $k$ range in which the proposed eOMP significantly differs from the original OMP. The result shows when $k\in [40,70]$, the proposed eOMP can achieve up to more than $20\%$ L0-norm reduction than OMP. 
\begin{figure}
\centering
\includegraphics[width=0.85\linewidth]{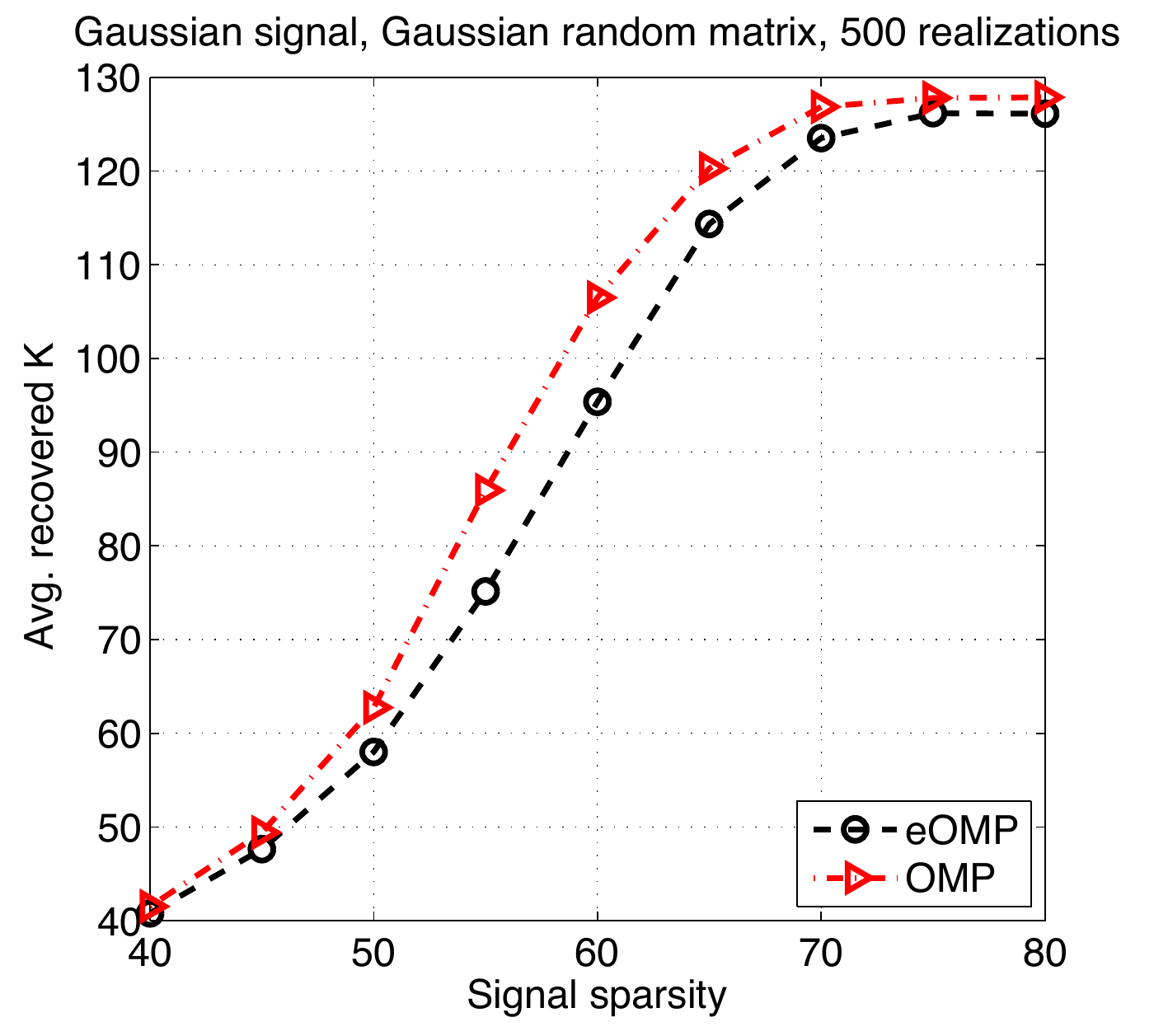}
\caption{Recovered sparsity for Gaussian sampling matrix}
\label{fig:sparserecov_gauss}
\end{figure}

Figure~\ref{fig:sparserecov_odct} shows the recovered sparsity of eOMP and OMP for ODCT dictionaries. In all 3 ODCT configurations, we consistently see a large sparsity saving comparing to OMP. The sparsity saving in ODCT cases is larger than the Gaussian ensemble case, which implies our proposed eOMP is more efficient for finding sparse solution when the sampling matrix $\mathbf{\Phi}$ is more mutually coherent. We note that it is well-knwon that practical audio and image signals can be represented efficiently using DCT and ODCT. Therefore, this is a very encouraging result. 
\begin{figure}
\centering
\includegraphics[width=0.85\linewidth]{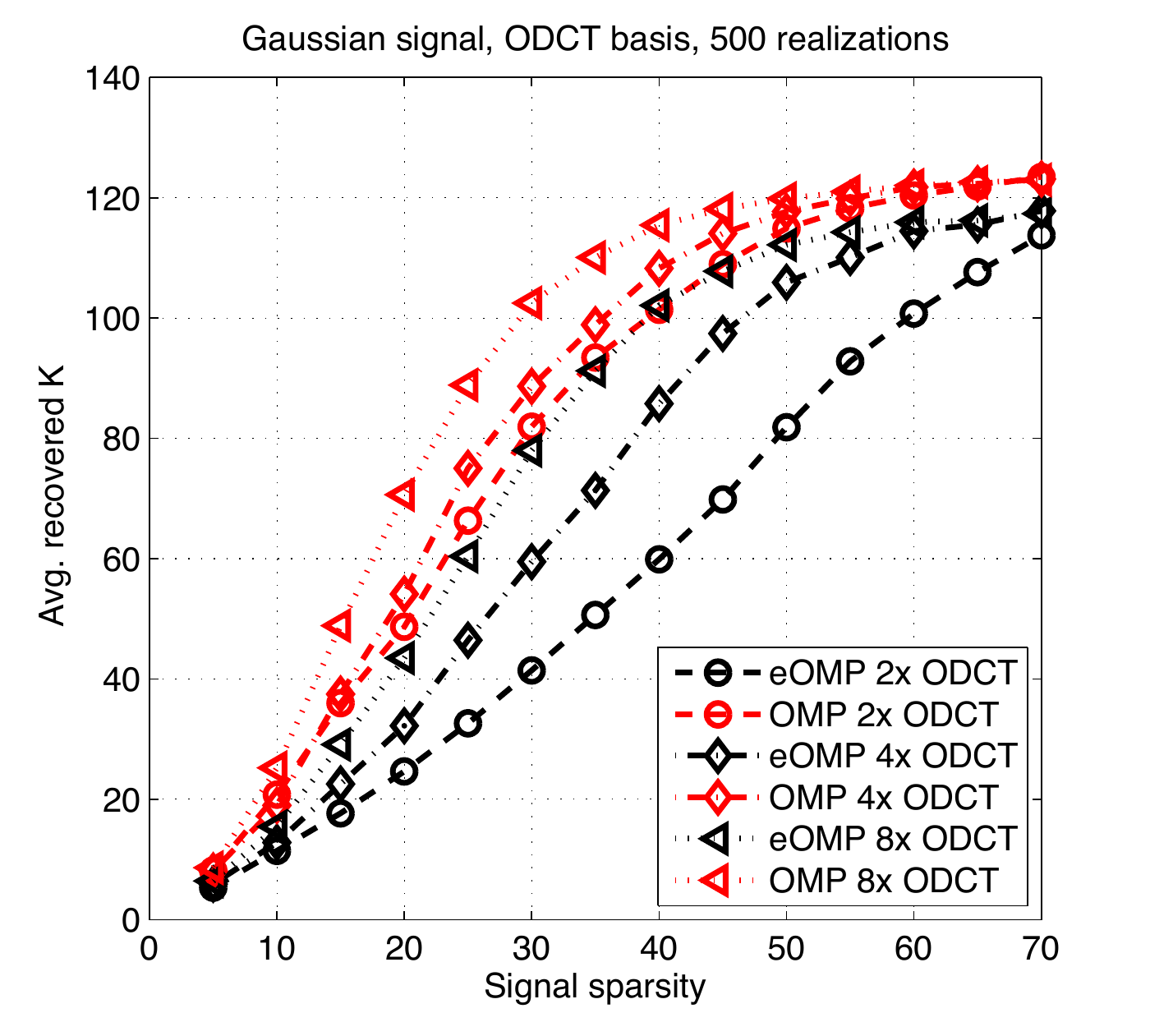}
\caption{Recovered sparsity for overcomplete DCT basis}
\label{fig:sparserecov_odct}
\end{figure}
\subsection{Real simulations}\label{sec:exp_real}
In this section, we show a simulation by adopting our proposed eOMP on a realistic application. We examine the gain of the proposed eOMP over the original OMP in the video sparse representation proposed in~\cite{xue2014video}. The essence of~\cite{xue2014video} is to find a sparse representation of a video frame block, by using a self-adaptive redundant dictionary consisting of all shifted blocks in the previous frame over a large search range. The motivation for using such a dictionary to represent a block is that the current block is likely to be very close to a few atoms in the dictionary.  In the special case when the current block is exactly equal to a shifted block in the previous frame due to motion, the sparsity is 1. In the  example examined here, the block size is $16\times16$, and the search range of $(-23\sim 24)$ shift in both horizontal and vertical direction is used, leading to a dictionary of size $N = 16\times 16, M = 2304$. Furthermore, we do not know the true sparsity of each block.

We show the result as PSNR vs. K plots, where the vertical is the recovery fidelity measured by PSNR, and the horizontal is the average L0-norm over all the blocks in a frame. To generate a curve, we vary the recovery fidelity threshold $\epsilon$ to approximate different expected distortions of applying uniform quantization on the coefficients with respect to the orthonormalized atoms. Assuming the bits spent on each non-zero entry is roughly the same, the higher curve would indicate better quality for the same bit rate. Figures~\ref{fig:stockholm} and~\ref{fig:football} show PSNR vs. K for two video sequences called \verb+Stockholm+ and \verb+Football+, respectively. By using the proposed algorithm, the recovered K for the same reconstruction fidelity (i.e. same PSNR) reduces by up to $30\%$, which translates to substantially reduced bit rate. This is impressive improvement as this simulation was ran on real signals (images) with a real dictionary that is highly redundant and correlated.
\begin{figure}
\centering
\includegraphics[width=0.8\linewidth]{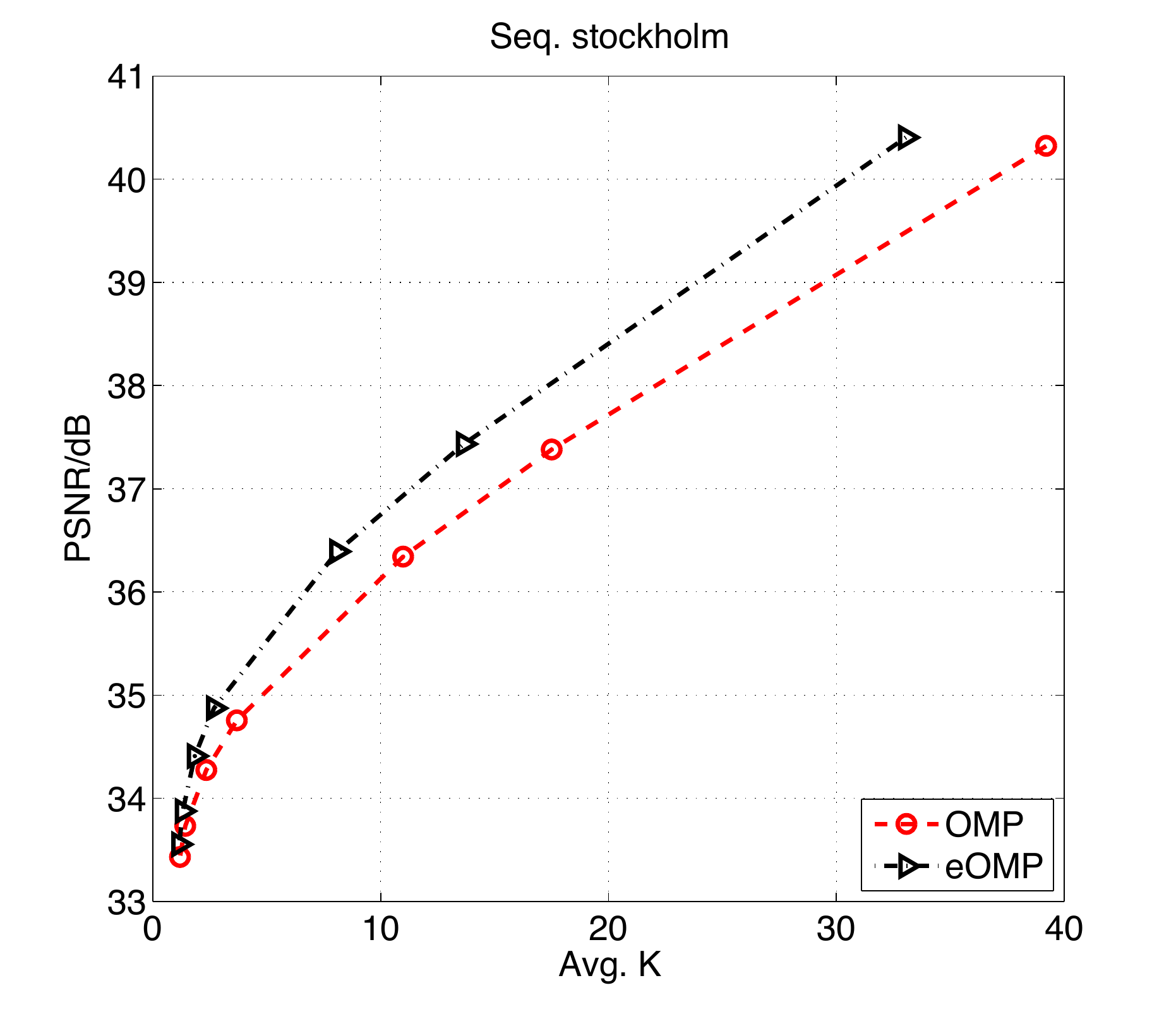}
\caption{PSNR vs. K for Seq. Stockholm}
\label{fig:stockholm}
\end{figure}
\begin{figure}
\centering
\includegraphics[width=0.8\linewidth]{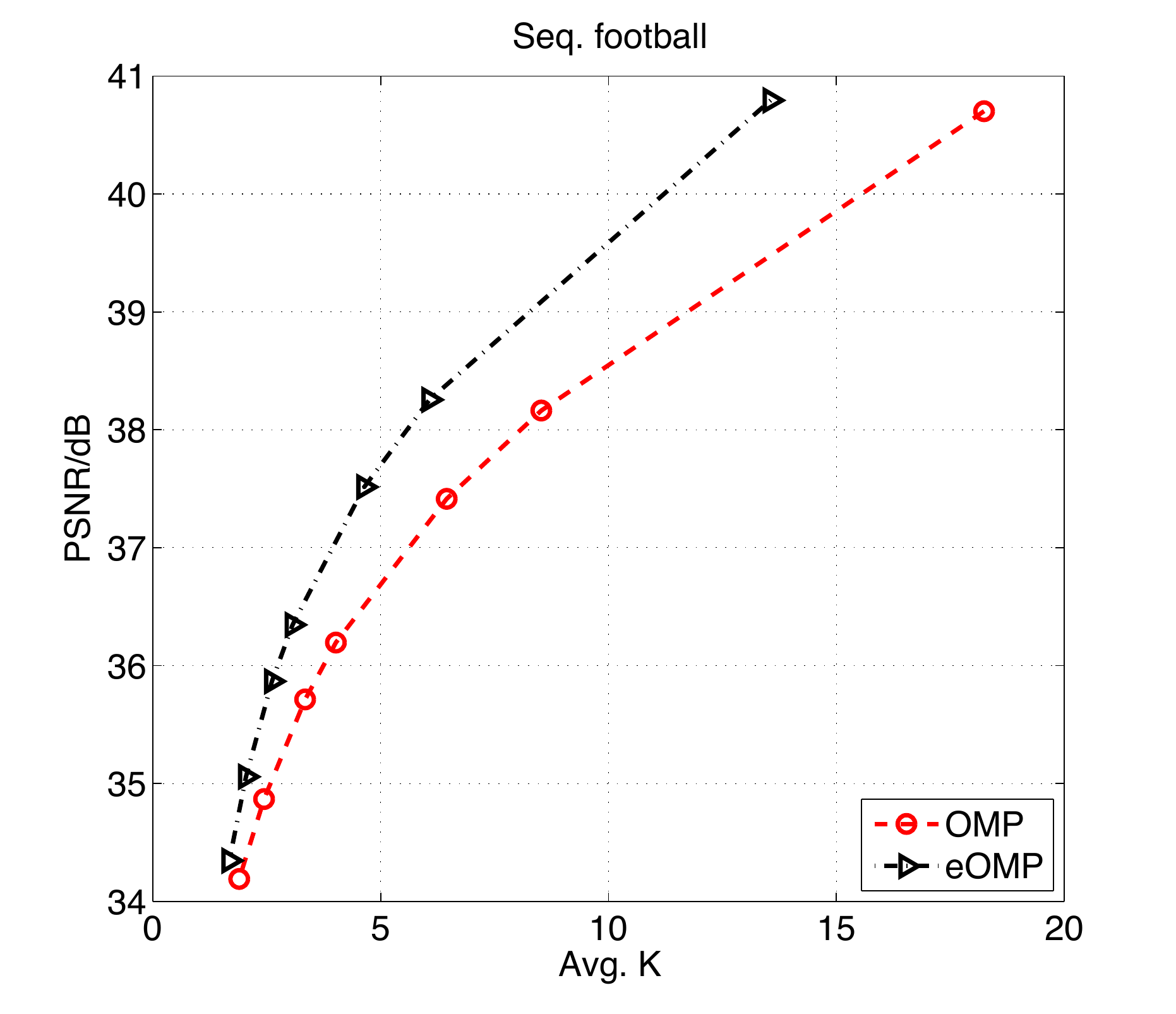}
\caption{PSNR vs. K for Seq. Football}
\label{fig:football}
\end{figure}
\section{Complexity Analysis}\label{sec:comp}
The major difference between the original OMP (implemented using Algorithm~\ref{alg:embedOMP}) and eOMP algorithms is the need of recursively performing one-step orthonormalization for all the remaining atoms. Although it sounds tedious, the complexity analysis we will show here implies manageable complexity increase. First, without loss of generality, assume the majority of the complexity coming from multiplications. Therefore an element-wise multiplication of a vector in $\mathbb{R}^{N}$ is of complexity $\mathcal{O}(N)$. The one-step orthonormalization shown in Eq.~\ref{eq:onestepGS} includes two inner products and two scaling. To perform a one-step orthonormalization, the total complexity is $\mathcal{O}(4N)$. Note that this has to be done for each remaining atom.

The rest of the ingredients are shared among two algorithms. The correlation calculation between the residual and each dictionary atom takes a complexity of $\mathcal{O}(N)$; the residual/solution update takes a complexity of $\mathcal{O}(2N)$ in each iteration. For OMP using Algorithm~\ref{alg:embedOMP}, orthonormalization one atom with respect to $k$ previous chosen atoms has a complexity of $\mathcal{O}\left((2k+2)N\right)$. Assuming no other numerical trick is used, this complexity is similar to that of solving a least square update by QR decomposition.

Now, consider the case that we run both algorithms on a sampling matrix of $\mathbb{R}^{N\times K}$, with signal in $\mathbb{R}^{N}$ for $s$ iterations. For original OMP algorithm in Algorithm~\ref{alg:embedOMP}, note that the set of chosen atoms is always increasing, so the orthonormalization process takes a complexity of $\mathcal{O}\left(\left(1+\ldots+s-1\right)2N+2sN\right)$ for $s$ iterations. Therefore the total complexity is:
\begin{align}\label{eq:ompcomp}
\mathcal{O}&\left(\left(K+K-1+\ldots+K-s+1\right)N\right)+ \nonumber \\
\mathcal{O}&\left(\left(1+\ldots+s-1\right)2N+2sN\right)+\mathcal{O}\left(2sN\right) \nonumber \\
& = \mathcal{O}\left(\left(\frac{(2K-s+7)s}{2}+s(s-1)\right)N\right)
\end{align}
For the eOMP, one-step orthonormalization is applied to all remaining atoms, which is decreasing in number, so the second term is replaced by $\mathcal{O}\left(\left(K-1+\ldots+K-s+1\right)4N\right)$ for $s$ iterations. Therefore the total complexity is:
\begin{align}\label{eq:revompcomp}
\mathcal{O}&\left(\left(K+K-1+\ldots+K-s+1\right)N\right)+ \nonumber \\
\mathcal{O}&\left(\left(K-1+\ldots+K-s+1\right)4N\right)+ \mathcal{O}\left(2sN\right) \nonumber \\
& = \mathcal{O}\left(\left(\frac{(2K-s+5)s}{2}+2(2K-s)(s-1)\right)N\right)
\end{align}
Assuming $K\gg s$, Eq.~\ref{eq:ompcomp} and Eq.~\ref{eq:revompcomp} are approximately $\mathcal{O}(NKs)$ and $\mathcal{O}(5NKs)$, respectively. The ratio of two complexity terms is constantly around 5. To validate this, we measure the total running time for the $2\times$ overcomplete DCT simulation shown in Sec.~\ref{sec:exp_syn} for both the original OMP and the eOMP. For the original OMP, we use the implementation provided in SparseLab~\cite{donoho23sparselab} software, which solves the least squares update using QR decomposition.
\begin{table}
\begin{adjustwidth}{-.2in}{}
\centering
\caption{Total running time (in seconds)}
\label{tab:runtime}
\vspace{-.2in}
\footnotesize{
\begin{tabular}{|c|c|c|c|c|c|c|}
\hline
Signal sparsity & 10 & 20 & 30 & 40 & 50 & 60 \\ \hline
OMP in SparseLab~\cite{donoho23sparselab} & 0.8614 & 2.393 & 4.076 & 5.506 & 6.418 & 6.678 \\ \hline
OMP in Algorithm~\ref{alg:embedOMP} & 0.6078 & 1.709 & 2.889 & 3.866 & 4.542 & 4.666 \\ \hline
Proposed eOMP & 5.119 & 11.30 & 17.45 & 25.64 & 33.71 & 39.79 \\ \hline
\end{tabular}
}
\vspace{-.15in}
\end{adjustwidth}
\end{table}
Table~\ref{tab:runtime} lists the total running time for all three algorithms, implemented in Matlab, for 500 trials of each given K. We see the ratios are around 5$\times$. The slightly larger ratio may be due to the inefficient implementation of our eOMP algorithm.
\section{Conclusion}\label{sec:conc}
In this work, we propose an improved OMP algorithm eOMP, by recursively orthonormalizing the remaining atoms. We have proved the eOMP algorithm maximizes the residual reduction at each iteration. Through extensive numerical simulations with both synthetic and real data, we conclude the proposed eOMP algorithm can consistently achieve better recovered sparsity and higher exact recovery rate than the original OMP. The gain is more significant when the signal's actual sparsity is unknown and the sampling matrix
(dictionary) is highly correlated. Our complexity analysis shows that the proposed eOMP algorithm has the same order of magnitude of complexity as the original OMP, roughly equals 5 times more complexity than that of the original OMP when the sparsity is significantly smaller than the dictionary size.

The work reported in this paper is by no means complete: we have yet to provide a mathematical proof on the exact recovery condition, and we are looking to possible further ways to reducing the running time. There are several recent theoretical developments on the greedy L0-norm algorithms but as far as we known, none of them are designed with a clear application perspective. For example, the algorithm needs to know the actual signal sparsity, which is typically unknown in practical application. Our goal is to pursue a better, more application friendly L0-norm minimization algorithm.

%



\section*{Acknowledgment}
The authors would like to thank Mr. Yi Zhou for providing simulation results shown in Fig.~\ref{fig:stockholm}-\ref{fig:football}.

\ifCLASSOPTIONcaptionsoff
  \newpage
\fi



\bibliographystyle{IEEEtran}
\bibliography{revOMP}




\end{document}